\documentclass{article}[12pt]
\usepackage[top=1.5in, bottom=1.5in, left=1.5in, right=1.5in]{geometry}
\usepackage{graphicx}
\usepackage{cite}
\usepackage{color}
\usepackage{subfig}
\usepackage{amsfonts,amsmath,amsthm}

\newtheorem{theorem}{Theorem}[section]

\newtheorem{prop}[theorem]{Proposition}
\newtheorem{conj}[theorem]{Conjeture}

\def\QED{\ensuremath{{\square}}}

\begin{document}


\title{$(n,m)$-Fold Covers of Spheres}

\author{Imre~B\'ar\'any \thanks{Alfr\'ed R\'enyi Mathematical Institute,
Hungarian Academy of Sciences,
Budapest, Hungary
and
Department of Mathematics,
University College London.
Partially supported by ERC Advanced Research Grant no 267165 (DISCONV), 
and by Hungarian National Research Grant K 83767.
{\tt barany.imre@renyi.mta.hu}}
\and Ruy~Fabila-Monroy \thanks{Departamento de Matem\'{a}ticas,
        Cinvestav, D.F. M\'exico, M\'{e}xico. Partially supported by grant 153984 (CONACyT, Mexico).
        {\tt ruyfabila@math.cinvestav.edu.mx} }
\and Birgit Vogtenhuber \thanks{Institute for Software Technology,
        University of Technology, Graz, Austria.
        {\tt bvogt@ist.tugraz.at}}}

\maketitle

\begin{abstract}
A well known consequence of the Borsuk-Ulam theorem is that if the $d$-dimensional
sphere $S^d$ is covered with less than $d+2$ open sets, then there is a set containing
a pair of antipodal points. In this paper we provide lower and upper bounds
on the minimum number of open sets, not containing a pair of antipodal points, needed to cover 
the $d$-dimensional sphere $n$ times, 
with the additional property that the northern hemisphere is covered $m > n$ times. 
We prove that if the open northern hemisphere is to be covered $m$ times then at least
$\left \lceil \frac{d-1}{2} \right \rceil+n+m$ and at most $d+n+m$ sets
are needed. For the case of $n=1$ and $d \ge 2$, this number is equal
to $d+2$ if  $m \le \left \lfloor \frac{d}{2}\right \rfloor + 1$ and 
equal to $\left \lfloor \frac{d-1}{2} \right \rfloor + 2 +m$  if $m > \left \lfloor \frac{d}{2}\right \rfloor + 1$.
If the closed northern hemisphere is to be covered $m$ times then 
$d+2m-1$ sets are needed, this number is also sufficient. 
We also present results on a 
related problem of independent interest. We prove
that if $S^d$ is covered $n$ times with open sets, not containing a pair
of antipodal points, then there exists a point that is covered at least  
$\left \lceil \frac{d}{2}\right \rceil +n$ times. Furthermore,  we show that 
there are covers in which no point is covered more than $n+d$ times.

\end{abstract}
\section{Introduction}

The Lusternik-Schnirelmann~\cite{LS} version of the Borsuk-Ulam theorem states that in any covering
of the $d$-dimensional sphere $S^d$ with at most $d+1$ open sets, there
is a set containing a pair of antipodal points. A natural generalization of this result has been introduced
by Stahl~\cite{stahl}; he considered \emph{$n$-fold covers}, in which every point of $S^d$ must 
be covered at least $n$ times. He showed that in every $n$-fold cover of $S^d$ with at most 
$d+2n-1$ open sets, there is a set containing a pair of antipodal points. Using a construction
of Gale~\cite{gale}, it can easily be shown that this bound is tight. For every $n \ge 1$, Gale
constructed a set of $d+2n$ points on the $d$-dimensional unit sphere,
 with the property that every open half-space that contains the origin contains at least $n$ points of the set. Placing an open hemisphere with its
pole at each of these points provides an $n$-fold cover of $S^d$ with
$d+2n$ open sets, in which no set contains a pair of antipodal points. We refer to this cover
as the \emph{Gale $n$-fold cover} of $S^d$. An $n$-fold cover of $S^d$ is said to
be \emph{antipodal} if none of its sets contains a pair of antipodal points.

We consider a variation on this theme. Let $m > n \ge 1$. An \emph{$(n,m)$-fold} cover of $S^d$
is an $n$-fold cover of $S^d$, in which every point of the \emph{open} northern hemisphere
is covered $m$ times. An \emph{$\overline{(n,m)}$-fold} cover of $S^d$
is an $n$-fold cover of $S^d$, in which every point of the \emph{closed} northern hemisphere
is covered $m$ times. Let $f(d,n,m)$ be the minimum number of sets in an antipodal 
$(n,m)$-fold cover of $S^d$ with open sets. In a similar way, let $\overline{f}(d,n,m)$ be 
the minimum number of sets in an antipodal $\overline{(n,m)}$-fold cover of $S^d$ with open sets. 
Since the case of $d=0$ is trivial, we will always assume that $d\ge 1$.

In this paper we show lower and upper bounds on $f(d,n,m)$ (Theorem~\ref{thm:f}) and provide
the exact value of $\overline{f}(d,n,m)$ (Theorem~\ref{thm:f_c}). 
We also compute the exact value of $f(d,1,m)$ (Theorem~\ref{thm:f_1_m} and Proposition~\ref{prop:f_1_1_2}).
The search for a lower bound of $f(d,n,m)$ lead us to study the problem of finding
a point covered many times in an antipodal $n$-fold cover of $S^d$ with open sets.
Let then $Q(d,n)$ be the maximum integer such that in every antipodal $n$-fold cover of $S^d$ with open sets there exists a 
point that is covered $Q(d,n)$ times. This paper is organized as follows. In Section~\ref{sec:Q} we show
upper and lower bounds on $Q(d,n)$ (Theorem~\ref{thm:Q}). In Section~\ref{sec:f} we give
our results on $f(d,n,m)$ and $\overline{f}(d,n,m)$.

\section{Bounds on $Q(d,n)$}\label{sec:Q}
 
The problem of determining $Q(d,1)$ has been studied before. Its exact value
of $Q(d,1)=\left \lfloor \frac{d}{2} \right \rfloor+2$
has been settled in a series of papers by \v{S}\v{c}epin~\cite{scepin}, 
Izydorek and Jaworowski~\cite{ant}, and Jaworowski~\cite{periodic}. 
An explicit cover yielding the upper bound for $Q(d,1)$ was given by Simonyi and Tardos~\cite{circular}.
They started by covering $S^d$ with the projections from the origin
of the closed facets of a regular $(d+1)$-simplex. Afterwards, they
replaced the points of these sets that were covered more than 
$\left \lfloor \frac{d}{2} \right \rfloor +1$ times with a new closed set.
(Although this gives a cover with closed sets, sufficiently small open
neighborhoods of these sets give the desired cover.)
At first glance, it seems sensible to use a similar idea to upper bound $Q(d,n)$.
However, our attempts of cutting out neighborhoods of often-covered points of an $n$-fold cover,
and placing patches of $n$ sets instead, always produced points that were covered
an excessive amount of times. Finally, we decided to use Gale's $n$-fold cover of $S^d$.

Ky Fan's theorem~\cite{fan} can be used to prove a lower bound of $\left \lceil \frac{d}{2} \right \rceil+1$
for  $Q(d,1)$. For proving a lower bound of $Q(d,n)$, we will use the following reformulation of Ky Fan's 
theorem that has been presented in~\cite{circular}.

\begin{theorem}\textbf{(Ky Fan's Theorem.)}\label{thm:fan}

Let $\mathcal{F}$ be an antipodal cover of $S^d$. Assume that a linear order is given on $\mathcal{F}$.
Then there exist $F_1 < F_2 < \cdots< F_{d+2}$ sets of $\mathcal{F}$ such
that \[ F_1\cap-F_2\cap F_3\cap -F_{4}\cap \cdots \cap (-1)^{d+1}F_{d+2}\neq \emptyset.\]
\end{theorem}
\begin{flushright} \QED \end{flushright}

We now give our bounds on $Q(d,n)$.

\begin{theorem}\label{thm:Q}
$\left \lceil \frac{d}{2}\right \rceil +n \le Q(d,n) \le d+n$.
\end{theorem}
\begin{proof}

First we prove the lower bound.
Let $\mathcal{F}$ be an antipodal $n$-fold cover of $S^d$ with open sets.
The intersections of all intersecting subsets of $\mathcal{F}$ consisting of $n$ sets
form an antipodal $1$-fold cover $\mathcal{F}'$ of $S^d$ with open sets.
Explicitly, 
$\mathcal{F}':=\{\bigcap \mathcal{C}:\mathcal{C}\subset \mathcal{F}, 
|\mathcal{C}|=n \textrm{ and } \bigcap \mathcal{C} \neq \emptyset\}$.
For an arbitrary linear order of $\mathcal{F}$, every set  of $\mathcal{F}'$ is of the form
$C=\bigcap_{i=1}^n F_i$ for some $F_1 < F_2 < \cdots < F_n$ in $\mathcal{F}$.
Hence, we may assign the tuple $v(C):=(F_1,F_2,\dots,F_n)$ to $C$ and define
a linear order on $\mathcal{F}'$, by setting $C_1 < C_2$ if and only if
$v(C_1) < v(C_2)$ in the lexicographical order of the tuples $v(C_1)$ and $v(C_2)$.
By Ky Fan's theorem there exist sets $C_1< C_2< \cdots < C_{d+2}$ of $\mathcal{F}'$ 
such that $\bigcap_{i=1}^{d+2} (-1)^{i-1}C_i$ is not empty. 
 Since $\mathcal{F}'$ is an antipodal cover, the first coordinates 
of the tuples associated to consecutive $C_i$'s 
are different; by the additional assumption that $v(C_1) < \dots < v(C_{d+2})$, all of these first 
coordinates are also pairwise different. Let $x \in \bigcap_{j=1}^{\lceil (d+2)/2 \rceil} C_{2j-1}$.
This point is in all the first coordinates (sets) of the tuples $v(C_{2j-1})$, and in all
the coordinates (sets) of the last tuple.
 As all of these sets are different, $x$ is in at least 
$\lceil (d+2)/2 \rceil+n-1= \left \lceil \frac{d}{2} \right \rceil +n$
different sets of $\mathcal{F}$.

The upper bound is given by Gale's $n$-fold cover of $S^d$. In this cover a point $x \in S^d$ is not covered 
by precisely those hemispheres (sets) whose poles are separated from $x$ by the hyperplane
through the origin and orthogonal to $\vec{x}$. Since there are at least $n$ of these
sets and there are $d+2n$ sets in this cover, $x$ is covered at most $d+n$ times.
\end{proof}

We conjecture that the upper bound of Theorem~\ref{thm:Q} is tight.

\begin{conj} \label{con:Q}
$Q(d,n)=d+n$ for $n \ge 2$.
\end{conj}

\section{Bounds on $f(d,n,m)$ and $\overline{f}(d,n,m)$. } \label{sec:f}

In this section we prove our results for $f(d,n,m)$ and $\overline{f}(d,n,m)$.
We start by showing the exact values of $\overline{f}(d,n,m)$ and $f(d,1,m)$.

\begin{theorem}\label{thm:f_c}
$\overline{f}(d,n,m) = d+2m-1$ for $m>n$.
\end{theorem}
\begin{proof}
Let $\mathcal{F}$ be an antipodal $\overline{(n,m)}$-fold cover of $S^d$ with
open sets. 
Note that the intersection of $\mathcal{F}$ with the equator is an antipodal $m$-fold cover of $S^{d-1}$
with open sets.
Therefore, as shown by Stahl~\cite{stahl}, $|\mathcal{F}|\ge d+2m-1$, which proves the lower bound. 
For the upper bound, rotate Gale's $m$-fold cover of $S^d$ so that one of the hemispheres (sets) in this cover
coincides with the southern hemisphere.  Remove this hemisphere to obtain 
an antipodal $\overline{(n,m)}$-fold cover of $S^d$ with $d+2m-1$ open sets.
\end{proof}

\begin{figure}
  \begin{center}
   \includegraphics[width=0.5\textwidth]{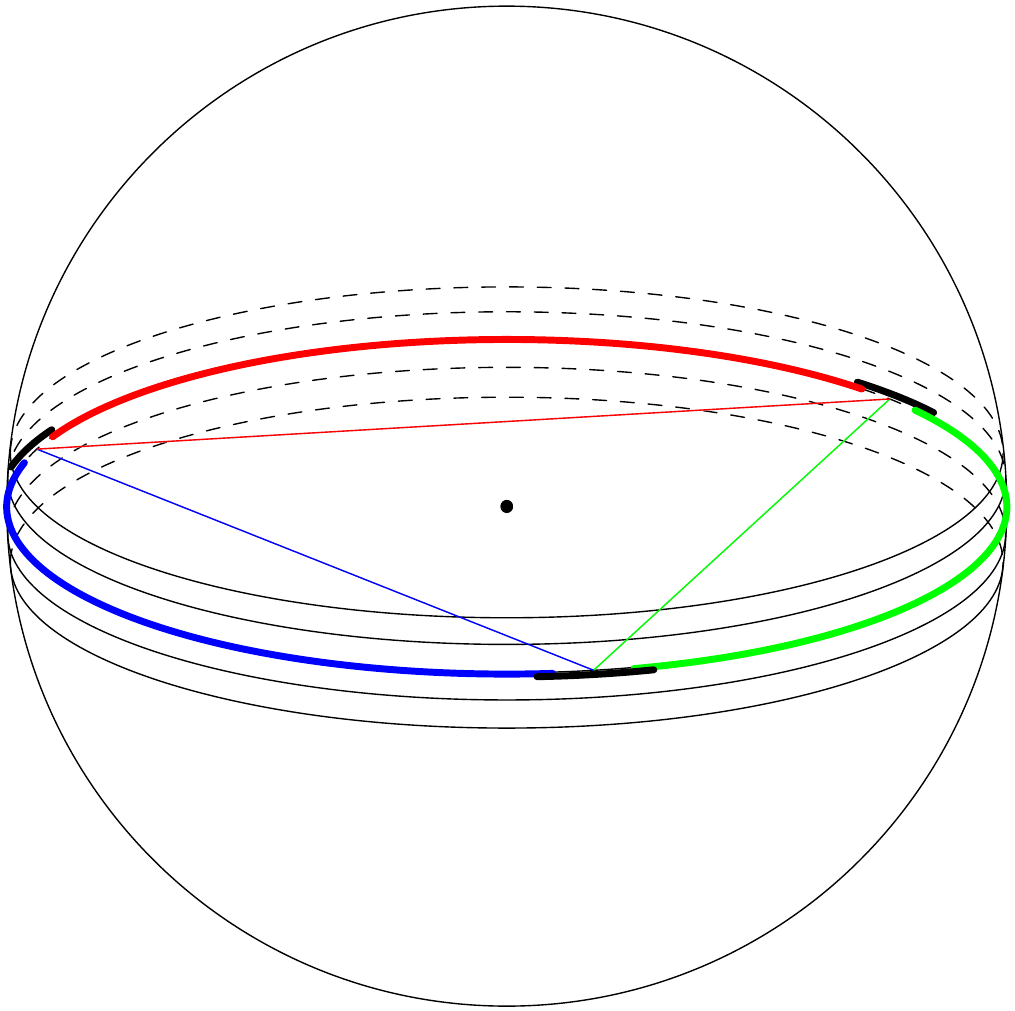}
\end{center}
\caption{The $1$-fold cover of the equator in the proof of Theorem~\ref{thm:f_1_m}.}
\label{fig:equator}
\end{figure}

\begin{theorem}\label{thm:f_1_m}
For $d \ge 2$, $f(d,1,m)$ is equal to:

\begin{equation*}
  f(d,1,m) =
  \begin{cases}
    d+2 & \text{if }  m \le \left \lfloor \frac{d}{2}\right \rfloor + 1,\\
    \left \lfloor \frac{d-1}{2} \right \rfloor + 2 +m & \text{if } m \ge \left \lfloor \frac{d}{2}\right \rfloor + 1.
  \end{cases}
\end{equation*}
\end{theorem}
\begin{proof}
First we prove the lower bound. Let $\mathcal{F}$ be an antipodal
$(1,m)$-fold cover of $S^d$ with open sets. Since $\mathcal{F}$ covers
the equator once, there is a point
in the equator of $S^d$ covered at least $Q(d-1,1)=\left \lfloor \frac{d-1}{2} \right \rfloor+2$ 
times. Just below this point there is a point in the southern hemisphere covered by the same sets; its antipodal
point in the northern hemisphere is covered by at least $m$ other sets.
Thus there are at least $\left \lfloor \frac{d-1}{2} \right \rfloor+2+m$ sets in $\mathcal{F}$.
This lower bound is tight for $ m \ge \left \lfloor \frac{d}{2}\right \rfloor + 1$. For $m < \left \lfloor \frac{d}{2}\right \rfloor + 1$, the
 better lower bound of $d+2$ follows immediately from Lusternik-Schnirelmann theorem~\cite{LS} and the fact that 
$\mathcal{F}$ is a $1$-fold cover of $S^d$.

For the upper bound we carefully construct an antipodal 
$\left (1,\left \lfloor \frac{d}{2}\right \rfloor + 1 \right )$-fold 
cover of $S^d$ with $d+2$ open sets. This cover proves the tight upper bound of $d+2$ for
$m \le \left \lfloor \frac{d}{2}\right \rfloor + 1$. For $m > \left \lfloor \frac{d}{2}\right \rfloor + 1$,
we add $m-\left \lfloor \frac{d}{2}\right \rfloor - 1$ open northern hemispheres
to this cover to produce an antipodal $(1,m)$-fold cover of $S^d$ with 
$ \left \lfloor \frac{d-1}{2} \right \rfloor + 2 +m$ open sets.

Assume that $S^d$ is the unit sphere centered at the origin. 
We start by constructing a $1$-fold cover of its equator $S^{d-1}$ 
(the intersection of $S^d$ with the hyperplane $x_{d+1}=0$) in the following way.
Fix a regular $d$-simplex $\tau$ centered at the origin and having its vertices
on $S^{d-1}$. Project its closed facets from the origin to $S^{d-1}$ and 
let $F_1', \dots, F_{d+1}'$, be these projections. This produces an antipodal
cover of $S^{d-1}$ with $d+1$ closed sets. Let $D'$ be the set of points
of $S^{d-1}$ that are covered at least $\left \lceil \frac{d+2}{2} \right \rceil=\left \lceil \frac{d}{2} \right \rceil+1$
times in this cover.  Note that $V'$ is closed, and since there are only $d+1$
sets $F_i'$, it does not contain a pair of antipodal points.
Choose $\varepsilon_2 > \varepsilon_1 > 0$. 
Let $D$ be the open $\varepsilon_2$-neighborhood of $D'$, in $S^{d-1}$.
Further, let $F_i$ be the open $\varepsilon_1$-neighborhood of $F_i'\setminus D$, also in $S^{d-1}$.
Choose $\varepsilon_2$ small enough such that none of $F_1, F_2,\dots, F_{d+1}$ and $D$
contain a pair of antipodal points.
Choose $\varepsilon_1$ small enough with
respect to $\varepsilon_2$, so that every point $x \in S^{d-1}$  has an open
neighborhood that is covered by at most $\left \lceil \frac{d}{2} \right \rceil$ of the sets $F_i$.
Then $\mathcal{F}:=\{F_1, F_2,\dots, F_{d+1}, D\}$  is an antipodal $1$-cover of $S^{d-1}$ with $d+2$ open sets. 
See Figure~\ref{fig:equator} for an illustration of the $d=2$ case.

We now extend $\mathcal{F}$ to an antipodal $\left (1,\left \lfloor \frac{d}{2}\right \rfloor + 1 \right )$-fold cover of $S^d$ with open sets.
Roughly speaking, we first extend all sets of $\mathcal{F}$ to parts of a ``belt''. Afterwards, we further 
extend the ``facet'' sets $F_i$ of $\mathcal{F}$ to cover the northern hemisphere 
and the set $D$ of $\mathcal{F}$ to cover the southern hemisphere. 
Let $\pi$ be the orthogonal projection of $\mathbb{R}^{d+1}$
to the hyperplane $x_{d+1}=0$. Let $ 0 < \delta_1' < \delta_2' < 1$.
Let $C_i''$ be the set of points $x \in \mathbb{R}^{d+1}$ with $\pi(x) \in F_i$, whose last coordinate
satisfies $-\delta_2' < x_{d+1} < \delta_2'$. Similarly, let $C_{d+2}''$ be the set of points 
$x \in \mathbb{R}^{d+1}$ with $\pi(x) \in D$, 
whose last coordinate  satisfies $-\delta_2' < x_{d+1} < \delta_1' $. 
Note that while the facet sets $F_i$ are extended symmetrically to the north and the south,
the set $D$ is extended further to the south than to the north.

Next, for each $1\leq i \leq d+2$, let $C_i'$ be the set of points $x \in S^d$ such that
the infinite ray with apex in the origin and passing through $x$ intersects $C_i''$. The
sets $C_i'$ the parts of ``belts'' on the sphere. Let $\delta_i = \sin\tan^{-1}(\delta_i')$ 
be the corresponding heights of the belt parts, for $i=1,2$. 

Finally, for $1 \le i \le d+1$, let $C_i$ be the union of the open
northern hemisphere with $C_i'$, minus the closure of $-C_i'$. Further, let 
$C_{d+2}$  be the union of the southern hemisphere with $C_{d+2}'$, minus the closure of $-C_{d+2}'$.
See Figures~\ref{fig:north} and~\ref{fig:south} for an illustration of
the northern and southern hemispheres respectively (for the case $d=2$ and $m=2$).
By construction, all these sets are open and do not contain a pair of antipodal points
of $S^d$.  What remains to show is that $\mathcal{C}:=\{C_1,\dots,C_{d+2}\}$ is in fact a 
$\left(1,\left \lfloor \frac{d}{2}\right \rfloor + 1 \right)$-cover
of $S^d$. 

We distinguish a few cases with respect to the value of the last coordinate of the points of $S^d$.
The set of points of $S^{d}$ whose
last coordinate is greater than $\delta_2$ are covered $d+1$ times by $C_1, \dots, C_{d+1}$.
The set of points of $S^{d}$ whose last coordinate is less than or equal to $-\delta_2$ are covered once by $C_{d+2}$.

To each point $x \in S^d$ whose last coordinate satisfies
$-\delta_2 < x_{d+1} \le \delta_2$, we assign the point $x'$ 
in the equator such that $\pi(x)$ lies on the infinite ray from the origin to $x'$.
Suppose that the last coordinate of $x$ is greater than zero. 
By our choice of $\varepsilon_1$ and $\varepsilon_2$, there is an open neighborhood
of $x'$ in $S^{d-1}$ that is covered by at most $\left \lceil \frac{d}{2} \right \rceil$ sets of $-F_1, \dots, -F_{d+1}$.
Assume without loss of generality that $-F_1, -F_2,\dots, -F_{\left \lfloor \frac{d}{2} \right \rfloor+1}$ do not cover this neighborhood, 
then at least $C_1, C_2,\dots, C_{\left \lfloor \frac{d}{2} \right \rfloor+1}$ cover $x$. Therefore $\mathcal{C}$ covers
the open northern hemisphere $\left \lfloor \frac{d}{2} \right \rfloor+1$ times.  If the last coordinate of $x$ is greater than $-\delta_2$
and at most zero, then $x$ is covered by the extensions of the sets in $\mathcal{F}$
that cover $x'$. Thus, $x$ is covered at least once and $\mathcal{C}$ is a $\left (1,\left \lfloor \frac{d}{2}\right \rfloor + 1 \right )$-cover of $S^d$.
\end{proof}

\begin{figure}
\centering
\begin{tabular}{ccc}
\subfloat{\includegraphics[width=0.25\textwidth]{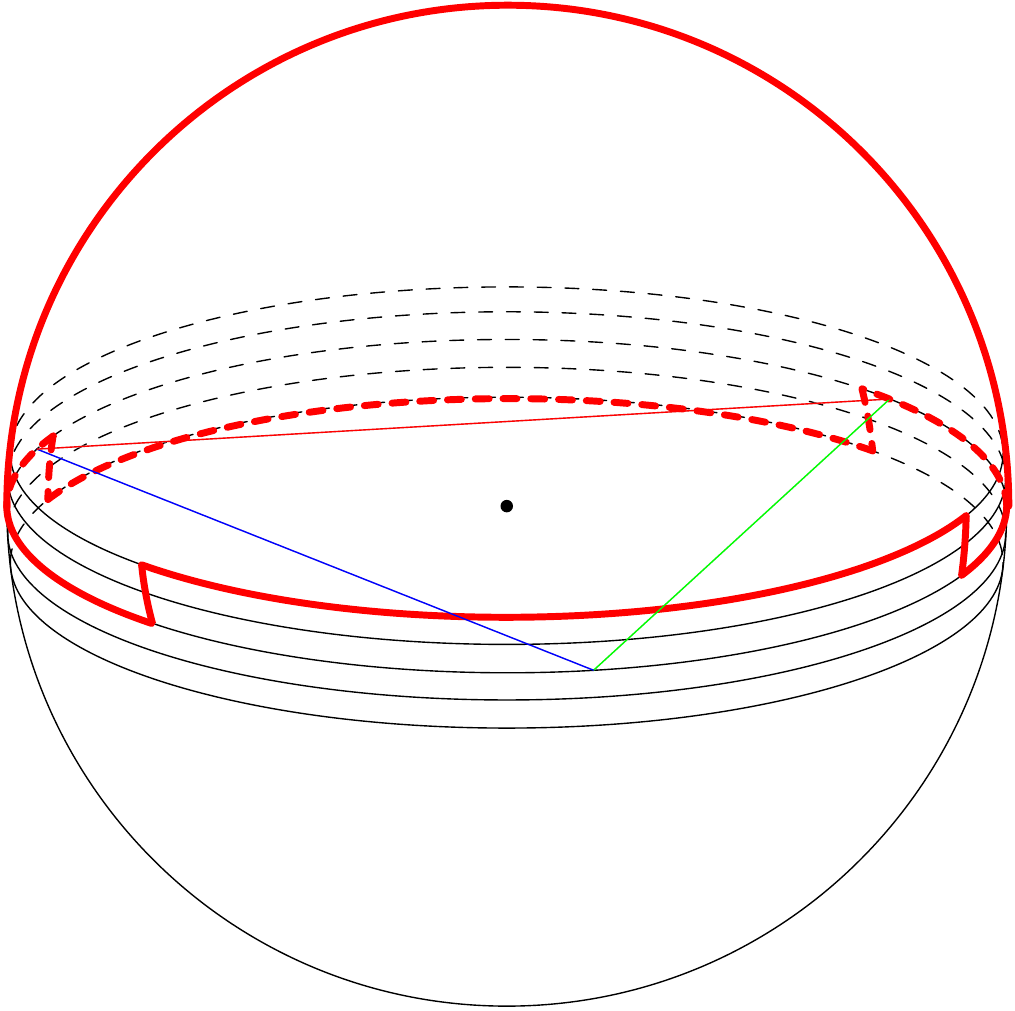}}
   & \subfloat{\includegraphics[width=0.25\textwidth]{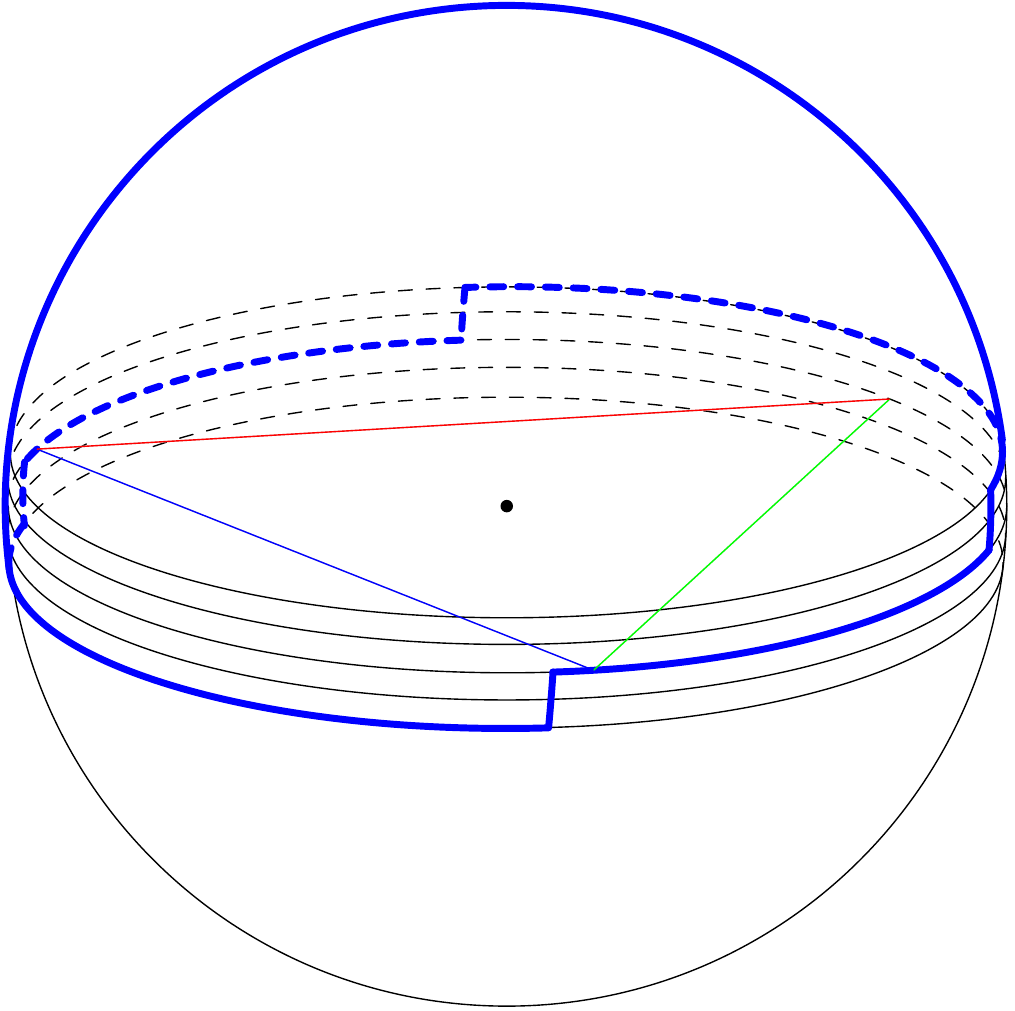}}
   & \subfloat{\includegraphics[width=0.25\textwidth]{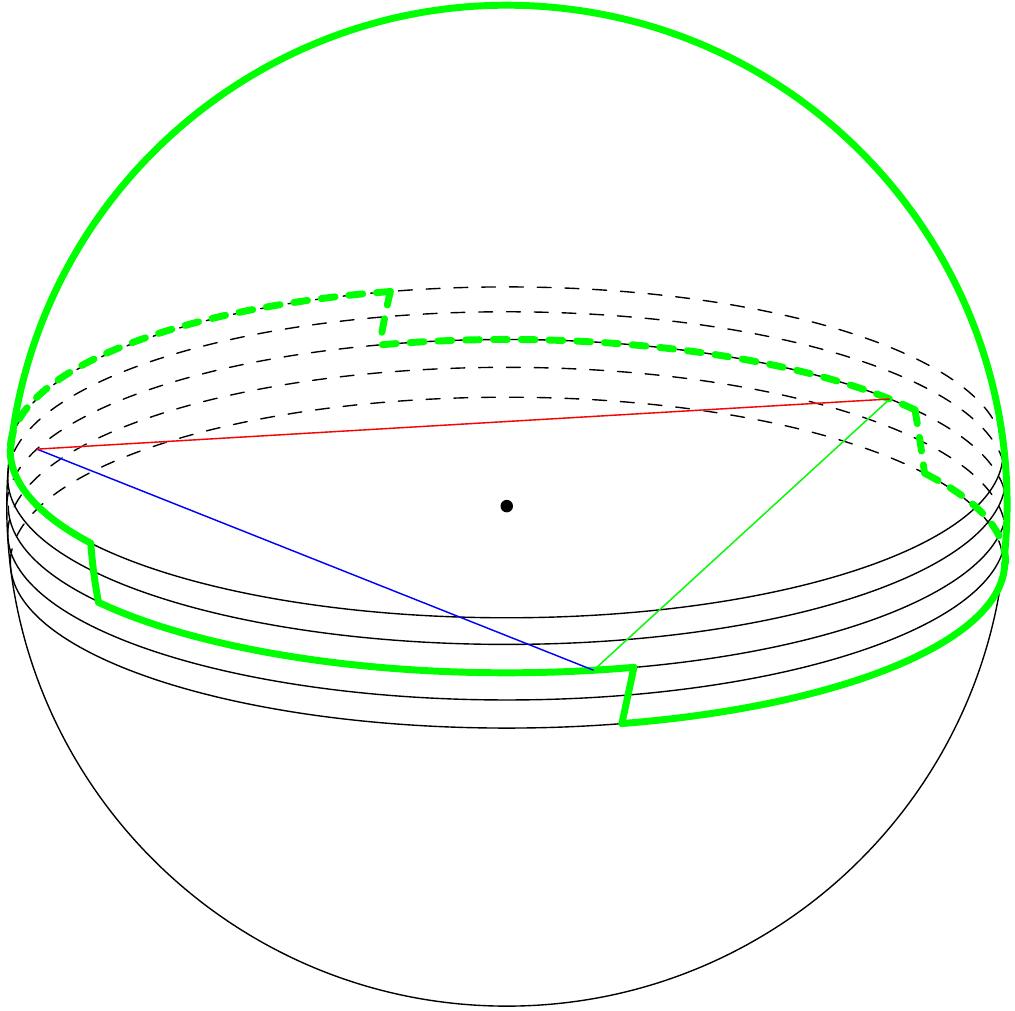}}\\
\subfloat{\includegraphics[width=0.25\textwidth]{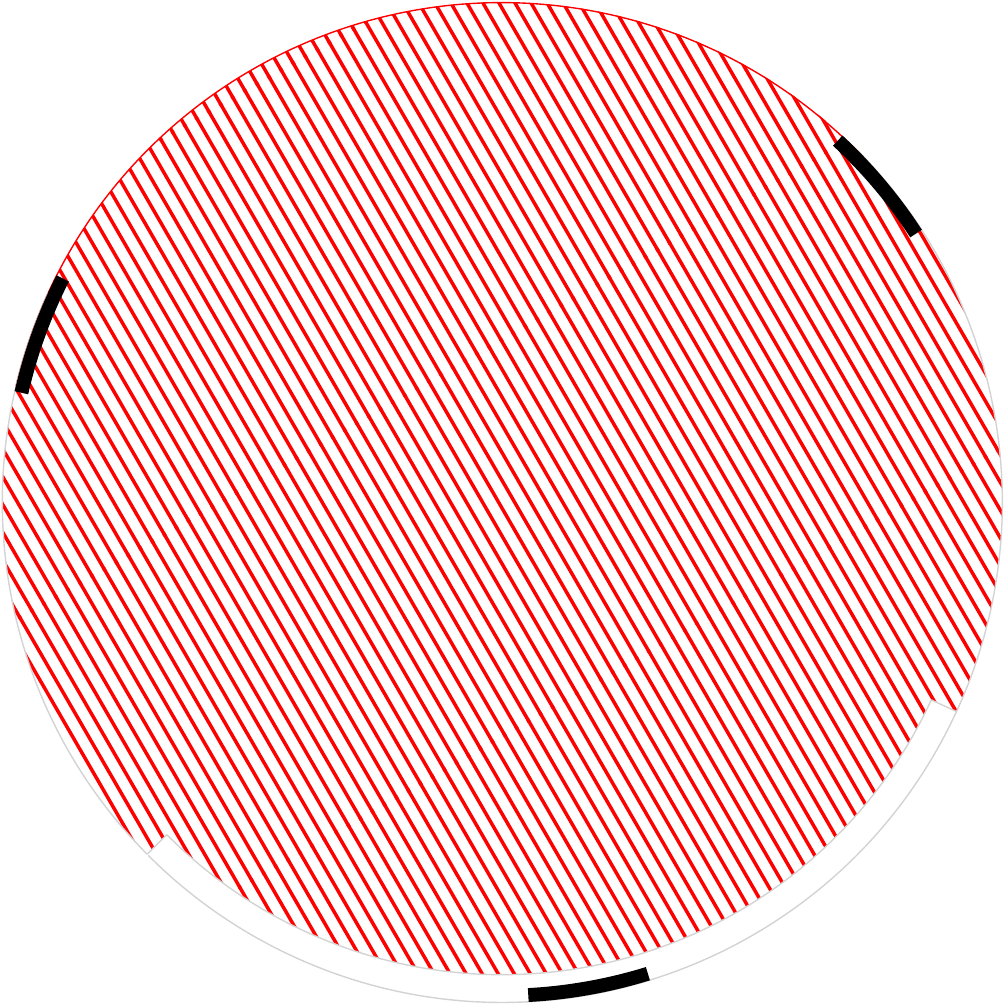}}
   & \subfloat{\includegraphics[width=0.25\textwidth]{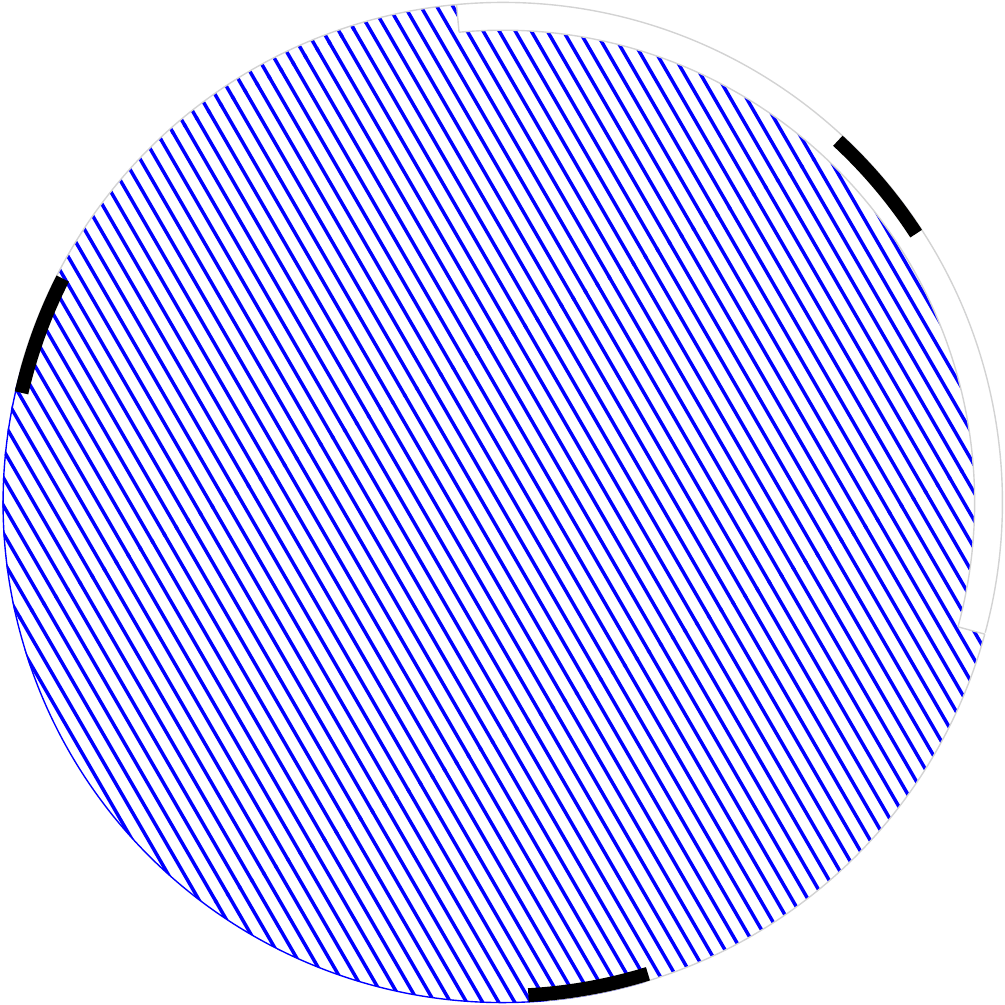}}
   & \subfloat{\includegraphics[width=0.25\textwidth]{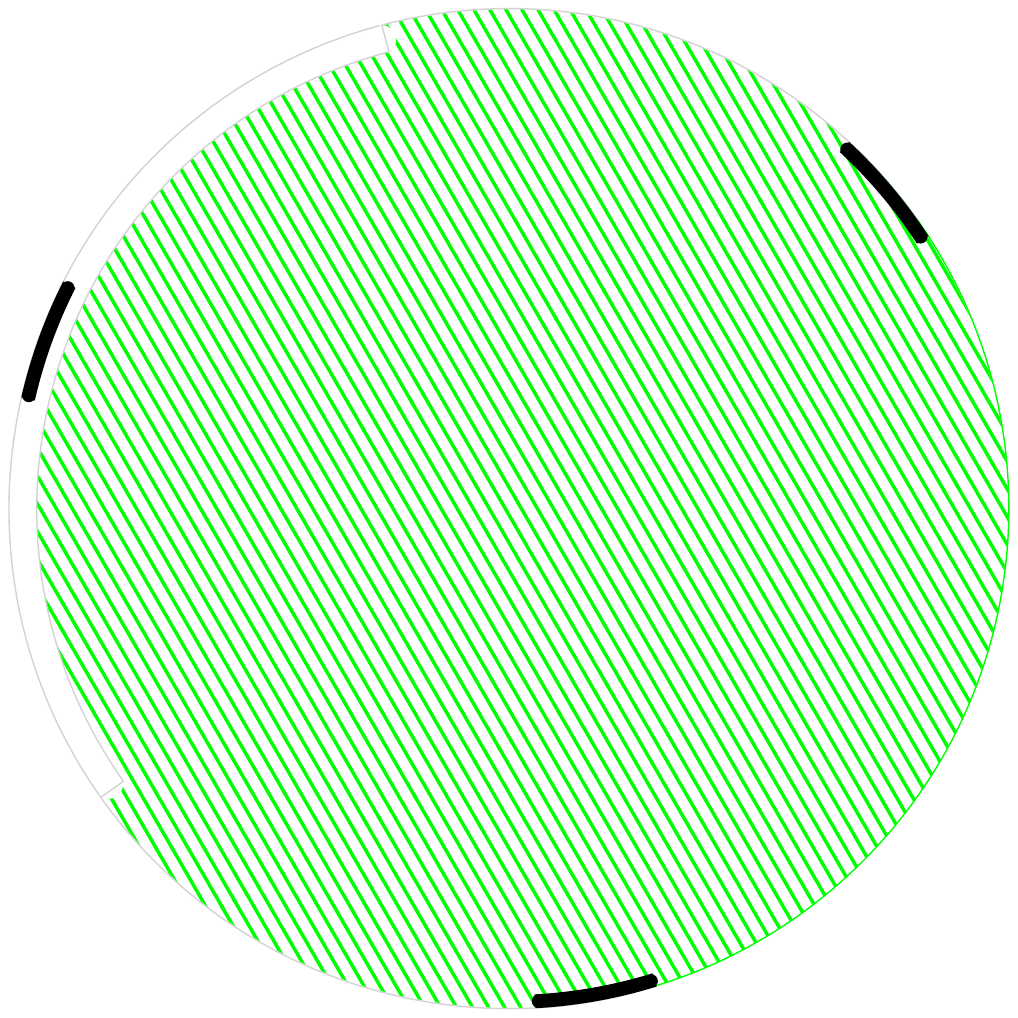}}
\end{tabular}
\caption{The covering of the northern hemisphere in the proof of Theorem~\ref{thm:f_1_m}}\label{fig:north}
\end{figure}

\begin{figure}
\centering
\begin{tabular}{cc}
\subfloat{\includegraphics[width=0.25\textwidth]{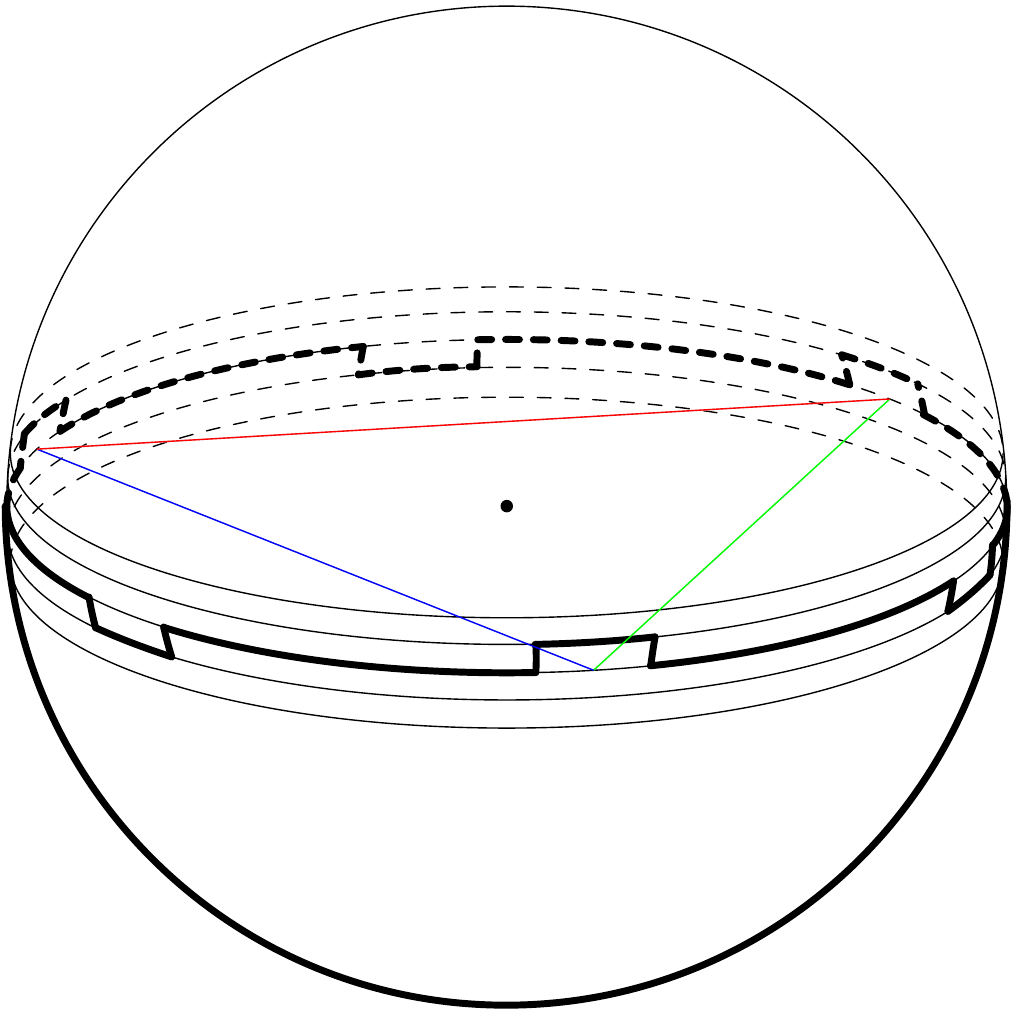}}
   & \subfloat{\includegraphics[width=0.25\textwidth]{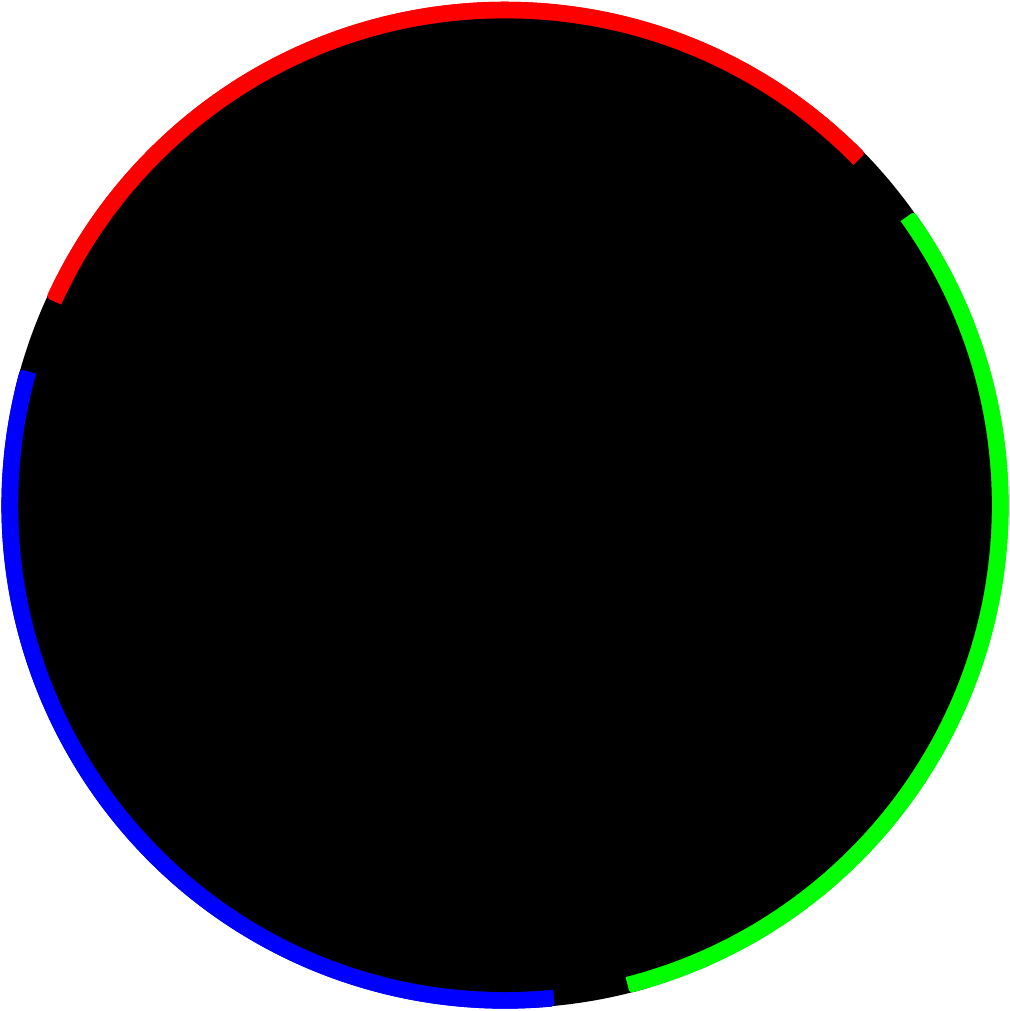}}
\end{tabular}
\caption{The covering of the southern hemisphere in the proof of Theorem~\ref{thm:f_1_m}}\label{fig:south}
\end{figure}

\begin{prop}\label{prop:f_1_1_2}
$f(1,1,m)=2+m$.
\end{prop}
\begin{proof}
For the upper bound, note that an antipodal $(1,m)$-cover of $S^1$
with $2+m$ open sets can be obtained by adding $m-1$ open northern hemispheres
to an antipodal $1$-cover of $S^1$ with three open sets.
For the lower bound,  
suppose first that there exist two points $x$ and $y$ on the open upper hemisphere that are not covered by the same $m$ sets.
Then on the path from $x$ to $y$ on the northern hemisphere, there exists a point $z$ that is covered by at least $m+1$ different sets.
As its antipodal point $-z$ needs at least one set to be covered as well, this gives a total of at least $m+2$ sets.
Hence, assume that all points on the open northern hemisphere are covered by the same $m$ sets. But then none of the two
points on the equator can be covered by any of these sets, again resulting in a total of at least $m+2$ sets.
\end{proof}

In the remaining part of this section, we present lower and upper bounds for $f(d,n,m)$ for arbitrary values of $n$ and $m$.

\begin{theorem}\label{thm:f}
$\left \lceil \frac{d-1}{2}\right \rceil+n+m \le f(d,n,m) \le d+n+m$.
\end{theorem}
\begin{proof}
Let $\mathcal{F}$ be an antipodal $(n,m)$-fold cover of $S^d$
with open sets. Since $\mathcal{F}$ covers the equator once, by Theorem~\ref{thm:Q}
there is a point in the equator of $S^d$ covered at least
$\left \lceil \frac{d-1}{2}\right \rceil+n$ times. Just below this point
there is a point in the southern hemisphere covered by the same sets; its 
antipodal point in the northern hemisphere 
is covered by at least $m$ other sets. Thus in total there
are at least $\left \lceil \frac{d-1}{2}\right \rceil+n+m$ sets
in $\mathcal{F}$. For the upper bound take Gale's $n$-fold cover of $S^d$
with $d+2n$ sets. Add $m-n$ open northern hemispheres to obtain an antipodal $(n,m)$-fold cover of $S^d$
with $d+n+m$ open sets. 
\end{proof}

In the case where $m-n < \left \lceil \frac{d}{2} \right \rceil$, the following proposition 
gives an improvement of the lower bound from Theorem~\ref{thm:f}. We omit the proof, which is a direct application 
of Stahl's result in combination with the fact that every $(n,m)$-fold cover is also an $n$-fold cover.

\begin{prop}
$f(d,n,m) \ge d+2n$.
\end{prop}

In the proof of the lower bound of $f(d,n,m)$ in Theorem~\ref{thm:f}  we used the lower bound of $Q(d,n)$ given in Theorem~\ref{thm:Q}. 
Hence, any improvement on the lower bound
of $Q(d,n)$ immediately improves the lower bound of
 $f(d,n,m)$. Assuming that Conjecture~\ref{con:Q} holds, the proof of Theorem~\ref{thm:f}  
would leave a gap of only one between the lower and upper bounds of $f(d,n,m)$.
  
\section*{Acknowledgments.}
We are indebted to L\'aszl\'o Z\'adori \cite{za} who asked whether
$f(2,1,2)$ equals four or five, a question that started this
piece of research.


\begin{thebibliography}{1}

\bibitem{fan}
K.~Fan.
\newblock A generalization of {T}ucker's combinatorial lemma with topological
  applications.
\newblock {\em Ann. of Math. (2)}, 56:431--437, 1952.

\bibitem{gale}
D.~Gale.
\newblock Neighboring vertices on a convex polyhedron.
\newblock In {\em Linear inequalities and related system}, Annals of
  Mathematics Studies, no. 38, pages 255--263. Princeton University Press,
  Princeton, N.J., 1956.

\bibitem{ant}
M.~Izydorek and J.~Jaworowski.
\newblock Antipodal coincidence for maps of spheres into complexes.
\newblock {\em Proc. Amer. Math. Soc.}, 123(6):1947--1950, 1995.

\bibitem{periodic}
J.~Jaworowski.
\newblock Periodic coincidence for maps of spheres.
\newblock {\em Kobe J. Math.}, 17(1):21--26, 2000.

\bibitem{LS}
L.~Lusternik and L.~Schnirelmann.
\newblock M\'ethodes topologiques dans les probl\`emes variationnels. I. Pt. Espaces \`a un nombre fini de dimensions. Translated from Russian by J. Kravtchenko.
\newblock Paris: Hermann \& Cie, 1934.

\bibitem{scepin}
E.~V. {\v{S}}{\v{c}}epin.
\newblock A certain problem of {L}. {A}. {T}umarkin.
\newblock {\em Dokl. Akad. Nauk SSSR}, 217:42--43, 1974.

\bibitem{local}
G.~Simonyi, G.~Tardos, and S.~T. Vrec{\'i}ca.
\newblock Local chromatic number and distinguishing the strength of topological
  obstructions.
\newblock{\em Trans. Amer. Math. Soc.}, 361:889--908, 2009.

\bibitem{circular}
G.~Simonyi and G.~Tardos.
\newblock Local chromatic number, {K}y {F}an's theorem and circular colorings.
\newblock {\em Combinatorica}, 26(5):587--626, 2006.

\bibitem{stahl}
S.~Stahl.
\newblock Reductions of {$n$}-fold covers.
\newblock {\em Proc. Amer. Math. Soc.}, 72(2):422--424, 1978.

\bibitem{za} L. Z\'adori, privite communication, April 2013.

\end{thebibliography}
\end{document}